\documentclass{amsart}

\usepackage{amssymb,amsmath,amsthm,epsfig,verbatim,fancyhdr,graphicx,afterpage}

\newtheorem{lemma}{Lemma}[section]
\newtheorem{prop}[lemma]{Proposition}
\newtheorem{thm}[lemma]{Theorem}
\newtheorem{cor}[lemma]{Corollary}
\newtheorem*{mainthm}{Theorem A}
\newtheorem*{cliff}{Clifford's Theorem}
\newtheorem*{maincor}{Corollary B}

\theoremstyle{definition}

\theoremstyle{remark}
\newtheorem{rmk}[lemma]{Remark}

\newcommand{\Pro}{\mathbb P}

\newcommand{\F}{\mathbb F}
\newcommand{\Z}{\mathbb Z}

\begin{document}

\title{On the algebra structure of some bismash products}

\author[Matthew C. Clarke]{MATTHEW C. CLARKE\\
Trinity College,\\
Cambridge, CB\textup{2 1}TQ\\
email: \texttt{m.clarke@dpmms.cam.ac.uk}}

\begin{abstract}

We study several families of semisimple Hopf algebras, arising as bismash products, which are constructed from finite groups with a certain specified factorization. First we associate a bismash product $H_q$ of dimension $q(q-1)(q+1)$ to each of the finite groups $PGL_2(q)$ and show that these $H_q$ do not have the structure (as algebras) of group algebras (except when $q =2,3$). As a corollary, all Hopf algebras constructed from them by a comultiplication twist also have this property and are thus non-trivial. We also show that bismash products constructed from Frobenius groups do have the structure (as algebras) of group algebras.\end{abstract}

\maketitle


\section{Introduction}
\label{intro}

\nocite{C} \nocite{JM} \nocite{KMM} \nocite{Mo} \nocite{JL}

We first consider arbitrary finite-dimensional semisimple Hopf algebras over an algebraically closed field $k$ of characteristic zero. One may consult \cite{Mo} as a reference for Hopf algebra notation and terminology. Until recently, the only known way of constructing such Hopf algebras was to take a bicrossed product $H=k^L \#^{\tau}_{\sigma} k F$ where $L$ and $F$ are finite groups with a matched pair structure, followed possibly by a comultiplication twist i.e., by conjugating the comultiplication map $\Delta_H$ by some element $\Omega \in H \otimes H$. Hopf algebras obtained in this way are called {\em group-theoretical}. (Construction of a non group-theoretical example was recently achieved in \cite{N}.) In the special case that $H$ is a bismash product (i.e., with trivial cocycles $\tau$, $\sigma$), it is shown in \cite{KMM} how to completely determine its simple module dimensions (and thus the isomorphism class of its algebra structure, since it is semisimple) using certain actions of $L$ and $F$ on each other, although the algebra structure of bismash products was previously described in \cite{Ta}. Hopf algebras $H'$ obtained from such a $H$ by a comultiplication twist certainly have the same algebra structure as $H$. Hence, information obtained about the algebra structure of $H$ applies also to this wider class of examples. It was shown in \cite{C} that an infinite family of bismash products constructed via factorizations of symmetric groups do not have the structure, as algebras, of group algebras. In the present paper we prove a similar result corresponding to factorizations of $PGL_2(q)$. In particular, we consider the factorization $PGL_2(q) = CS$ where $C \cap S = 1$ for any Singer cycle $C$ and stabilizer $S$ of a point on the projective line, and set $H_q= k^C \# k S$. Note that dim$H_q = |PGL_2(q)| = q(q - 1)(q + 1)$. Our main result is as follows.

\begin{mainthm} {\em (Theorem \ref{main})} Let $q=p^h$ be a prime power with $h \ge 1$. Then the algebra structure of $H_q$ is not isomorphic to a group algebra except when $q=2$ or $3$.

\end{mainthm}

\begin{maincor} {\em (Corollary \ref{maincor})} Any Hopf algebra $H'$ obtained from $H$ or $H^{*}$ by a comultiplication twist is non-trivial.

\end{maincor}

The method used herein is to show that the simple module dimensions of the bismash product in question are \[\underbrace{1, \ldots, 1}_{q-1},\ q-1,\  \underbrace{q, \ldots, q}_{q-1},\] \noindent and then use character theory to show that they do not correspond to the irreducible character degrees of any finite group. The exceptions $q=2$ and $q=3$ correspond to the group algebras of the symmetric groups $S_3$ and $S_4$ respectively.

\section{Bismash products and their representations}
\label{bismash}

Let $G$ be a finite group with subgroups $L$ and $F$. We call the triple ($G,L,F$) a {\em factorized group} if $L\cap F = 1$ and $G=LF$. (The order in which the subgroups appear will turn out to be important.) Adopting the notation of \cite{C}, we define group actions as follows. A left action $L\times F \rightarrow F$, $l : f \mapsto {}^{[l]}f$, and a right action $L\times F \rightarrow L$, $f : l \mapsto l^{[f]}$ determined by the rule \begin{equation}
lf = ({}^{[l]}f)(l^{[f]}). \label{newrule} \end{equation} 

The {\em bismash product} $k^L \# kF$ associated with the factorized group ($G,L,F$)  is the Hopf algebra with underlying vector space $k^L \otimes kF$ (we use \# instead of $\otimes$ to emphasize that the Hopf algebra structure is constructed differently from the usual tensor product), endowed with multiplication \[(l \# f)(\overline{l} \# \overline{f}) = \delta_{l^{[f]},\overline{l}}l \# (f \overline{f}),\] \noindent and unit \[u(\alpha) = \sum_{l \in L} (l \# \alpha).\]
\noindent The algebra structure of $(k^L \# kF)^{*} = k^F \# kL$ then dualizes to give a comultiplication \label{commult} \[\Delta (l \# f) = \sum_{\hat{l} \in L} (l\hat{l}^{-1} \# {}^{[\hat{l}]}f) \otimes (\hat{l} \# f).\] \noindent We also have a counit $\varepsilon(l \# f) = \delta_{l,1}$ and antipode $S(l \# f) = (l^{[f]})^{-1} \# ({}^{[l]}f)^{-1}$. In some of the literature bismash products are defined using matched pairs of groups, but it is an elementary fact (see, eg, \cite{JM}) that there is a bijection between these two classes of objects which respects the bismash product construction.

The following fact is crucial.

\begin{lemma}

Let {\em ($G,L,F$)} be a factorized group. Then the bismash product $k^L \# kF$ is semisimple.

\end{lemma}

This can be proved using a version of Maschke's theorem for Hopf algebras. Next, we describe how to calculate the dimensions of the simple modules of a bismash product, which, by the above, completely determines its $k$-algebra structure. This will provide us with a means of attacking the question of whether this structure is the same as some group algebra. For this we will need the following, which is essentially Corollary 3.5 of \cite{KMM}.

\begin{lemma}

Let {\em ($G,L,F$)} be a factorized group and let $l_1, \ldots,\ l_n$ be a complete set of representatives of the $F$-orbits in $L$. Denote the irreducible character degrees of $F_{l_i}= \ ${\em Stab}$_F(l_i)$ by $d_{i,1}, \ldots,\ d_{i,r_i}.$ Then the dimensions of the simple modules of $k^L \# kF$ are \begin{equation} \label{chardims}
d_{i,1}[F : F_{l_i}], \,\ldots\, ,\  d_{i,r_i}[F : F_{l_i}]\ \ \mbox{for } 1 \le i \le n.
\end{equation}
\end{lemma}

\noindent Note that we use the right action of $F$ on $L$ here. The other action is not needed.

\section{A family of Bismash products associated with $PGL_2(q)$}
\label{pgl}

Let $q=p^h$, where $p$ is prime and $h \ge 1$. If $H$ is subgroup of some general linear group then we shall denote its natural image in the corresponding projective general linear group by $H /\!\!\sim$. We now describe an infinite family of semisimple Hopf algebras, denoted by $H_q$, associated with $PGL_2(q)$, including their simple module dimensions.

It is well known that $GL_n(q)$ always contains cyclic subgroups of order $q^n - 1$. These, and their images in $PGL_n(q)$ (which have order $(q^n-1)/(q-1)$), are called {\em Singer cycles}. See \cite{Hu}, page 187 onwards, for more details. Let $C \le PGL_n(q)$ be a Singer cycle. Then $C$ acts regularly on the points of the projective space $\Pro^{n-1}(q)$. It follows that $PGL_n(q) = CS$ with $S \cap C = 1$, where $S$ is a point stabilizer of the natural action of $PGL_n(q)$ on $\Pro^{n-1}(q)$. We can thus consider the associated bismash product $k^C \# k S$, whose structure is completely determined by the factorized group $(PGL_2(q),C,S)$, as explained previously. (Note that this construction is independent of the choice of Singer cycle and/or point stabilizer. Indeed the Singer cycles comprise a complete set of conjugate subgroups, as do the point stabilizers. Replacing a factor by a conjugate in a factorized group gives rise to an isomorphic bismash product.) 

We now calculate the dimensions of the simple modules of the bismash product $H_q=k^C \# k S$ associated with $PGL_2(q)$. Let 

\begin{equation}
X^2 + \mu X + \lambda \label{mueq}
\end{equation}

\noindent be a primitive polynomial of degree $2$ over $\F_q$, and let

\[ x = \left( \begin{array}{cc}
0 & 1 \\
-\lambda & -\mu \end{array} \right).\]

\noindent One can check (by diagonalizing in $GL_n(q^n)$ for instance) that $\langle x \rangle$ is indeed a Singer cycle. Then $C = \langle x \rangle /\!\! \sim \  = \langle \overline{x} \rangle$  is a Singer cycle in $PGL_2(q)$. Together with the stabilizer $S$ of the point $[1:0]^T$,

\[S= \left\{ \left. \left( \begin{array}{cc}
a & b \\
0 & c \end{array} \right) \ \right| \ \begin{array}{c} 
a, c \in \F_{q}^{*} \\
b \in \F_q \end{array} \right\} \mbox{\bf \large /{\large $\sim$}},
\]

\noindent we have a factorized group $(PGL_n(q), C, S)$ from which we can construct the bismash product $k^C \# k S$. Recall that in order to calculate the dimensions of the simple modules we must consider the right action \[C \times S \rightarrow C,\ s:c \mapsto c^{[s]}\] defined by the rule $cs = s'c^{[s]}$, for some $s' \in S$. First we determine the orbits. Clearly $ 1 $ is an orbit on its own. The following completes the picture.

\begin{lemma} \label{transitive}

Let $u_1, u_2, \ldots, u_q$ denote the elements of \[U = \left\{ \left. \left( \begin{array}{cc}
d & t \\
0 & d \end{array} \right) \right| \ \begin{array}{c} 
d \in \F_{q}^{*} \\
t \in \F_q \end{array} \right\} \mbox{\bf \large /{\large $\sim$}} \ \ \le S.
\] \noindent Then $x^{[u_1]}, \ldots, x^{[u_q]}$ are distinct. In particular, the action of $S$ on $C^{\#}$ is transitive.

\end{lemma}

Next we must determine the character degrees of point stabilizers corresponding to these orbits.

\begin{lemma}

The irreducible character degrees of {\em \,Stab}${}_S(1) = S$ are $\underbrace{1, \ldots, 1}_{q-1},\ q-1$.

\end{lemma}

\begin{proof}

This follows from the fact that $S$ is a Frobenius group with Frobenius kernel \[K= \left\{ \left. \left( \begin{array}{cc}
a & b \\
0 & a \end{array} \right) \right| \ \begin{array}{c} 
a \in \F_{q}^{*} \\
b \in \F_q \end{array} \right\} \mbox{\bf \large /{\large $\sim$}}
\] \noindent and Frobenius complement \[H= \left\{ \left. \left( \begin{array}{cc}
d & 0 \\
0 & e \end{array} \right) \right| \ 
d, e \in \F_{q}^{*}  \right\} \mbox{\bf \large /{\large $\sim$}}.\qedhere
\]
\end{proof}

\begin{lemma} \label{stab}

{\em Stab}${}_S(\overline{x}) \cong \Z_{q-1}$.

\end{lemma}

\begin{proof}

First we make the observation that the action of $S$ on $C$ is isomorphic to its natural action on the projective line. A sufficient condition for this is that $S$ only has one conjugacy class of subgroups of order $q-1$. One may check this by hand or note that any subgroup of this order is a complement to $K$, and that by the Schur-Zassenhaus theorem  any two complements of a normal Hall subgroup are conjugate. \end{proof}

The above yields the following.

\begin{thm} \label{simpmods}

The dimensions of the simple modules of $H_q$ are

\[\underbrace{1, \ldots, 1}_{q-1},\ q-1,\  \underbrace{q, \ldots, q}_{q-1}.\]

\end{thm}

\section{Proof of the main result}
\label{proof}

In this section we prove Theorem A from the introduction.

\begin{thm} \label{main}

Let $q=p^h$ be a prime power with $h \ge 1$. Then the algebra structure of $H_q$ is not isomorphic to a group algebra except when $q=2$ or $3$.

\end{thm}

In order to prove this, we shall initially assume that there does exist a group (of order $q(q-1)(q+1)$) whose group algebra is isomorphic to $H_q$, and prove a series of lemmas leading to a contradiction. Clearly, by Theorem \ref{simpmods}, such a group would have irreducible character degrees \begin{equation}
\underbrace{1, \ldots, 1}_{q-1},\ q-1,\  \underbrace{q, \ldots, q}_{q-1}, \label{baddegrees}
\end{equation} \noindent and so we use character theory to show that this can not be the case. For a contradiction, we shall assume for the rest of this chapter that $G$ denotes a group with these irreducible character degrees, unless stated otherwise. The following lemma will be used several times.

\begin{lemma} \label{nabquotientL}

If $N$ is a non-identity normal subgroup of $G$ such that $G/N$ is nonabelian, then the character degrees of $G/N$ are \begin{equation}
\underbrace{1, \ldots, 1}_{q-1},\ q-1. \label{nabquotient}
\end{equation}

\noindent Moreover, $(G/N)'$ is an elementary abelian $p$-group.

\end{lemma}

\begin{proof}

Let $G$ be a group with character degrees (\ref{baddegrees}), such that $N$ is a normal subgroup with $G/N$ nonabelian. Suppose that $G/N$ has no irreducible character of degree $q-1$. It must, then, have a character of degree $q$ since it is nonabelian. Then we have that \[r + sq^2 =  |G/N| \equiv 0 \mod{q}, \] \noindent where $1 < r \le q-1$ is the number of linear characters, and $s \le q-1$ is the number of irreducible characters of degree $q$. This is impossible, and so we conclude that there must be at least one irreducible character of degree $q-1$.

If there is no irreducible character of degree $q$, then we have \[r + (q-1)^2 =  |G/N| \equiv 0 \mod{(q-1)}, \] \noindent which forces $r=q-1$. We claim that this is the only valid possibility. If there is an irreducible character of degree $q$, then \[r + (q-1)^2 + sq^2 = r + q^2 - 2q + 1 +sq^2 = |G/N| \equiv 0 \mod{q},\] \noindent which also forces $r=q-1$. But then \[q-1 + (q-1)^2 + sq^2 = |G/N| \equiv 0 \mod{(q-1)} , \] \noindent which implies $s=q-1$, and hence $N=1$.

Note that $G/N$ is soluble, since its derived group is a $p$-group, due to the number of linear characters in (\ref{nabquotient}). Clearly $N$ is maximal such that $G/N$ is nonabelian. Finally, note that $(G/N)'$ is a minimal normal subgroup of $G/N$, since it is contained in every non-identity normal subgroup of $G/N$. (In fact, this means that it is the unique minimal normal subgroup, although we do not need this fact.) Indeed, let us consider an arbitrary one, $K/N$, where $N < K$. $(G/N)' \subseteq G/N$ is equivalent to $(G/N)/(K/N) \cong G/K$ being abelian, which must be the case, by the maximality of $N$.  \end{proof}

Let $G$ be an arbitrary finite group, with $H \le G$. If $\chi$ denotes a character of $G$, then we denote the restricted character on $H$ by $\chi \!\downarrow_H$. We will need to make use of the following well-known result several times. (See, e.g., \cite{C2}.)

\begin{cliff} Let $G$ be a group with $N \lhd G$, and let $\chi$ be an irreducible character of $G$. Then \[\chi\!\downarrow_N = e\sum_{i=1}^{r} \xi_i,\] \noindent where the $\xi_i$ are distinct irreducible characters of $N$ of the same degree, and $e$ and $r$ both divide $[G:N]$.

\end{cliff}

The following result, due to I. M. Isaacs (Theorem 12.15 of \cite{ISA}), implies that any group with character degrees (\ref{baddegrees}) must be soluble with derived length three. The proof is fairly lengthy in that it employs several other special results from that text, but it can in fact be proved more directly in the case of (\ref{baddegrees}) by exploiting the fact that the two non-trivial character degrees are coprime. 

\begin{thm} \label{ISAthm}

A group with at most three distinct character degrees is soluble with derived length at most three.

\end{thm}

We will shortly need the following (easily verifiable) fact.

\begin{lemma} \label{JL}  Let $H$ be a subgroup of a finite group $G$. Let $\chi$ be an irreducible character of $G$, and let $\psi_1, \dots, \psi_s$ be the irreducible characters of $H$. Then writing $\chi\!\downarrow_H = d_1\psi_1 + \cdots + d_s\psi_s$, the non-negative integers $d_i$ satisfy \[\sum_{i=1}^{s} d_i^2\,  \le \, [G : H].\]
\end{lemma}

\begin{lemma} \label{Gprimelem}

The character degrees of $G'$ are \[\underbrace{1, \ldots, 1}_{q},\ q.\] \noindent Furthermore, $G''$ is abelian of order $q+1$.

\end{lemma}

\begin{proof}

Let $\chi$ be an irreducible character of $G$ of degree $q$. Then, by Clifford's theorem \[\chi\!\downarrow_{G'} = e\sum_{i=1}^{r} \xi_i,\] \noindent for some $e,r, \xi_i$. Moreover, $e$ and $r$ both divide $\chi (1) = q$ and $[G:G'] = q-\nolinebreak[4]1,$ so $e=\nolinebreak[4]r=1$. Hence $\chi\!\downarrow_{G'}$ is irreducible. Note that the restrictions of the characters of degree $q$ from $G$ to $G'$ must all agree because if $G'$ had more than one distinct irreducible character of degree $q$, then $\vert G' \vert \ge 2q^2$, which is impossible since $\vert G' \vert =\nolinebreak[4] q(q+1)$ by considering the number of linear characters of $G$.

Next, let $\Phi$ be the irreducible character of $G$ of degree $q-1$. Again, Clifford's theorem implies \begin{equation}
\Phi\!\downarrow_{G'} = d\sum_{i=1}^{s} \rho_i, \nonumber
\end{equation}\noindent for some $d,s,\rho_i$. We claim that \[\sum_{i=1}^{s} \rho_i^2 (1) \ge q-1.\] \noindent Indeed, first note that by Lemma \ref{JL}, $sd^2 \le [G:G'] = q-1$. Also, \[q-1 = \Phi\!\downarrow_{G'}(1) = d\sum_{i=1}^{s} \rho_i(1) = ds\rho_i(1);\] \noindent hence $d \le \rho_i(1)$ for each $i$. So \[\sum_{i=1}^{s} \rho_i^2 (1) \ge d\sum_{i=1}^{s} \rho_i (1) = \Phi\!\downarrow_{G'}(1) = q-1.\]

Finally, letting $\eta$ denote the principal character of $G'$, \[\eta(1) + \sum_{i=1}^{s} \rho_i^2 (1) + \chi\!\downarrow_{G'} (1) \ge 1 + (q-1) + q^2 = q(q-1) = \vert G' \vert,\] \noindent so we have all of the irreducible characters of $G'$.

If the $\rho_i$ were not linear, then $[G':G''] = 1$, and hence $G'$ would be perfect, which is not the case. So, actually, $G'$ has character degrees 

\[\underbrace{1, \ldots, 1}_{q},\ q,\]

\noindent as claimed. Finally, the number of linear characters is $q$, hence $[G':G''] = q$, and so $G''$ has order $q+1$, and is abelian since $G$ has derived length at most 3. \end{proof}

\begin{lemma} \label{q+1}

The irreducible character degrees of $G/G''$ are

\[\underbrace{1, \ldots, 1}_{q-1},\ q-1.\] \noindent Furthermore, $G''$ is a minimal normal subgroup of $G$. In particular, it is an elementary abelian $p$-group.

\end{lemma}

\begin{proof}

The character degrees follow from Lemma \ref{nabquotientL}. Let $N$ be a normal subgroup of $G$ such that $1 \not= N \le G''$. Then $G/N$ must be nonabelian; hence $G''=N$ by Lemma \ref{nabquotientL}, since the order of such a quotient is fixed. \end{proof}

We now prove the main result for $q$ even.

\begin{prop} \label{even}

Let $q=2^h$ with $h > 1$. Then there exists no group with character degrees \[\underbrace{1, \ldots, 1}_{q-1},\ q-1,\  \underbrace{q, \ldots, q}_{q-1}.\]

\end{prop}

\begin{proof}

Consider the conjugation map \[\gamma : G \longrightarrow Aut(G'').\] \noindent We start off by showing that $\ker\gamma = G''$, thus exhibiting a faithful representation of $G/G''$, since $Aut(G'') \cong GL_n(\hat{p})$ for some integer $n$ and prime $\hat{p}$, such that $\hat{p}^n=q+1$ by virtue of Lemmas \ref{Gprimelem} and \ref{q+1}. Let $K=\ker\gamma$.

First, we claim that $G/K$ cannot be abelian. Indeed, if it was, then $G' \le K$, which would mean that the conjugation action of $G'$ on $G''$ fixes each element of $G''$. The number of conjugacy classes of $G'$ would then have to be strictly greater than $\vert G'' \vert = q+1$. But the number of conjugacy classes of $G'$ is equal to the number of irreducible characters of $G'$, which we know to be $q+1$ by Lemma \ref{Gprimelem}.

To see that $K=G''$, note that \[G/K \cong (G/G'')/(K/G'');\] \noindent so the irreducible characters of $G/K$ correspond to those of $G/G''$ which have $K/G''$ in their kernel. It follows that since $G/K$ is nonabelian, it must have a non-linear irreducible character of degree $q-1$, hence $q-1$ divides $\vert G/K \vert$. Also $\vert G/K \vert \ge\nolinebreak[4] (q-\nolinebreak[4]1)^2 +\nolinebreak[4] 1$ by considering $\vert G/K \vert$ as the sum of the squares of its irreducible character degrees. It follows that $\vert G/K \vert = q(q-1)$. But, since $G''$ is abelian, $G'' \le K$, thus $K=G''$.

By embedding $GL_n(\hat{p})$ into $GL_n(\overline{\F_{\hat{p}}})$ we induce a faithful representation \[\gamma : G/G'' \hookrightarrow GL_n(\overline{\F_{\hat{p}}}),\] \noindent and provided $\hat{p}$ does not divide $\vert G/G'' \vert = q(q-1)$, the character degrees over $\overline{\F_{\hat{p}}}$ are the same as those over $k$.  But $\hat{p}^n = q+1$, which means $\hat{p}$ does not divide $q(q-1)$.

Let $\Phi$ be the character of $\gamma$, which is faithful because $\gamma$ is faithful. Then every irreducible character of $G/G''$ is a constituent of $\Phi^t$ for some $t$ (we mentioned this fact in the proof of Proposition \ref{ISAthm}). If $\Phi$ was a sum of linear characters only, then so would each power of it; thus the degree-($q-1$) irreducible character would never show up, which is a contradiction. It follows that the degree-($q-1$) irreducible character is a constituent of $\Phi$; thus $n \ge q-1$. But this is impossible since then $q+1=\hat{p}^n \ge \hat{p}^{q-1}$, while $q\ge 4$. \end{proof}

Note that the symmetric group $S_3$ has character degrees $1,1,2$ which corresponds to the case $q=p=2$. We now deal with $q$ odd.

\begin{prop} \label{odd}

Let $q=p^h$ be an odd prime power with $h \ge 1$. Then there exists no group with character degrees

\[\underbrace{1, \ldots, 1}_{q-1},\ q-1,\  \underbrace{q, \ldots, q}_{q-1},\]

\noindent except when $q=p=3$.

\end{prop}

\begin{proof}

By Lemmas \ref{q+1}  and \ref{Gprimelem}, $G''$ is elementary abelian of order $q+1$. Since $p$ is odd, it follows that $p^h + 1 = 2^n$. We claim that this forces $h=1$. First of all, $h$ must be odd. Indeed, if it were even, then $p^h$ would be a square and hence $p^h \equiv 1 \bmod{4}$, thus $2^n \equiv 2 \bmod{4}$. But then $n=1$,  which is impossible. To see that $h=1$ write \[2^n = p^h + 1 = (p+1)(p^{h-1} - p^{h-2} + p^{h-3} - \cdots - p + 1).\] \noindent The factor on the right comprises an odd number of odd summands and is therefore odd. It divides $2^n$ though, and so must be equal to 1. It follows that $h = 1$.

This leaves us with the case $q=p=2^n -1$, i.e., a Mersenne prime, which we deal with in a different manner. We can identify $G/G''$ with a subgroup of $GL_n(2)$ by the proof of Proposition \ref{even}. Since $(G/G'')' = G'/G''$ has order $p = 2^n -1$, it corresponds to a Singer cycle in $GL_n(2)$. In particular, the normalizer of a Singer cycle in $GL_n(2)$ contains $G/G''$, and thus has order at least $p(p-1)$. We will now prove that this is a contradiction.

Let $S=G'/G''$, $G=GL_n(2)$ and $G_2=GL_n(2^n)$. By diagonalizing a generator we can embed $G$ into $G_2$ such that $S=\langle y \rangle$, where \[ y= \left( \begin{array}{ccccc}
\xi & 0 & 0 & \cdots & 0 \\
0 & \xi^{r} & 0 & \cdots & 0 \\
0 & 0 & \xi^{r^{n}} & \cdots & 0 \\
\vdots & \vdots & \vdots & \ddots & \vdots \\
0 & 0 & 0 & \cdots & \xi^{r^{(n-1)}}  \end{array} \right) \] \noindent for some primitive element $\xi \in \F_{2^n}$. (The natural embedding followed by conjugation by a suitable element of $G_2$.) Now suppose that $n  \in N_{G_2}(S)$. Note that this is equivalent to $ny^kn^{-1} \in \langle y \rangle$ for all $k$, which is equivalent to $nyn^{-1} \in \langle y \rangle$. So in order to describe $N_{G_2}(S)$ it is sufficient to describe those $n = (n_{ij}) \in G_2$ such that $ny = y^tn$ for some $t$. Note the identities \begin{eqnarray}
(n_{ij})(y_{ij}) &=& (n_{ij}\xi^{r^{j-1}}) \nonumber \\ (y_{ij})^t(n_{ij}) &=& (n_{ij}\xi^{tr^{i-1}}).\nonumber\end{eqnarray} \noindent Assume these agree for some fixed $t$. If $n_{ij}$ is non-zero, then $\xi^{r^{j-1}} = \xi^{tr^{i-1}}$ which forces $t=r^{j-i}$. So $s=i-j$ is fixed for the non-zero entries of $n$. Conversely, any matrix, $n$, which has non-zero entries precisely in the $(ij)^{th}$ positions, where $i-j =s$ will do the trick. Now let us vary $t$. We see that $t$ must be a power of $r$, but that we are free to vary it subject to this restriction, and in so doing we vary $s$. It follows that $N_{G_2}(S)$ consists of all elements of the form \[ \left( \begin{array}{cccccccc}
			   &		  	  &         &            &               a_{1,s} &           &           &            \\
	  		 &  	  	 	&      &&   &  \ddots           &   &            \\
   		 	 &  	 	  	&        && &           &   a_{n,s-n}         &  \\
a_{s-1,1}&      		&       &&  &           &           &            \\
         &\ddots &         && &           &           &          \\
         & & a_{n,s-n}        && &           &           &    \end{array} \right); \] \noindent hence $|N_{G_2}(S)| = n(2^n-1)^n = np^n$. (The blank space indicates zeros.) Since $N_{G}(S) \le \nolinebreak[4] N_{G_2}(S)$, $|N_{G}(S)|$ divides $np^n$. Also $N_{G}(S) \le G$ so $|N_{G}(S)|$ divides 
         
         \[|G|\, =\, p(p-1)(p-3)(p-7)\cdots 3.\] \noindent It follows that $|N_{G}(S)|$ divides $pn$. Since $n < 2^n -2 = p-1$ for $n \ge 3$, \[|N_{G}(S)|\, \le\, pn\, <\, p(p-1),\] \noindent which is a contradiction. (In fact it is known that, in general, the normalizer of a Singer cycle $S$ is of the form $S \rtimes \Z_n$.)

Finally, recall that the symmetric group $S_4$ has character degrees $1,1,2,3,3$ which corresponds to the case $q=\nolinebreak p=\nolinebreak 3$.\end{proof}

Theorem \ref{main} now follows from Propositions \ref{even} and \ref{odd}.

\begin{cor} \label{maincor}

Any Hopf algebra $H'$ obtained from $H$ or $H^{*}$ by a comultiplication twist is non-trivial.

\end{cor}

\section{Some bismash products that are group algebras}

Given a factorized group $G=LF$ with $L\cap F = 1$, there are clearly two (in general, distinct) ways of creating a bismash product, depending on the order in which one takes the factors. (In fact, they are dual to one another.) Fix $H= k^L \# k F$. We consider the algebra structure of $H$ first in the case when $F$ is normal, and then in the case where $G$ is a Frobenius group with Frobenius kernel $L$. (A good reference for Frobenius groups is \cite{Go}. The definition appears on page 37, while the results of Chapter 10 give enough background for the material in this article, as well as other characterizations of Frobenius groups.) In particular, we show explicitly that it does have the structure of a group algebra in each of these cases. The following elementary lemma shows what happens if $F$ is normal.

\begin{lemma}

Let $(G,H,N)$ be a factorized group with $G = N \rtimes H$. Then the bismash product $k^H \# k N$ is isomorphic $($as a $k$-algebra$)$ to the group algebra $k(\Z_{|H|} \times\nolinebreak N)$.

\end{lemma}

\begin{proof}

Consider the right action of $N$ on $H$, which determines the module structure of $k^H \# k N$. Recall that this is given by $n : h \mapsto h^{[n]}$, defined by the rule $hn = n'h^{[n]}$ for some $n' \in N$. Since $N \lhd G$, we must have $h^{[n]} = h$ for all $n\in N$, $h\in H$. It follows that each point in $H$ is an orbit and each point stabilizer in $N$ is the whole of $N$. It follows from the above that the dimensions of the simple modules of $k^H \# k N$ consist of $|H|$ copies of the irreducible character degrees of $N$. The irreducible character degrees of $\Z_{|H|} \times\nolinebreak N$ agree with this.
\end{proof}

We now consider what happens if the normal factor forms the underlying set for the action. Let ($G,N,H$) be a factorized group with $G = N \rtimes H$ and consider the $k$-algebra structure of $k^N \# k H$. This is determined by the action $h : n \mapsto n^{[h]}$, defined by the rule $nh = h'n^{[h]}$ for some $h' \in H$. But since $N \lhd G$, $n^{[h]} = n^h$ and so this is just the conjugation action of $H$ on $N$. So, for a semidirect product $G=N \rtimes_{\varphi} H$, the associated homomorphism $\varphi : H \rightarrow Aut N$ which determines the group structure of $G$, given $N$ and $H$, is precisely the map which determines the $k$-algebra structure of the bismash product $k^N \# k H$.  (In essence, the advantage we have over arbitrary bismash products is that the image of this map is contained in  $Aut N$, rather than $S\!ym N$.) It should be possible, therefore, to make general statements about such bismash products by considering certain classes of semidirect products. One such class is where $\varphi$ is an injection onto a {\em fixed-point-free} subgroup of $Aut N$, i.e., one where only the identity element of $H$ fixes a point other than $1 \in N$. This is one characterization of a Frobenius group.

\begin{thm} \label{frobthm}

Let $(G,N,H)$ be a factorized group, such that $G$ is a Frobenius group with Frobenius kernel $N$ and Frobenius complement $H$. Consider the lower central series of $N$: \[N = \gamma_1(N) \rhd \gamma_2(N) \rhd \cdots \rhd \gamma_n(N) = 1;\]\noindent and let \begin{equation} \label{N*} N^{*} = \gamma_1(N) / \gamma_2(N) \times \gamma_2(N) / \gamma_3(N) \times \cdots \times \gamma_{n-1}(N) / \gamma_n(N).\end{equation} Then the bismash product $k^N \# k H$ is isomorphic (as a $k$-algebra) to the group algebra of a semidirect product $N^{*} \rtimes H$, which is also a Frobenius group, with Frobenius kernel $N^{*}$.
\end{thm}

\begin{proof}

First of all let us calculate the dimensions of the simple modules of $k^N \# k H$. By the remark above, the right action of $H$ on $N$ which determines the module structure of $k^N \# k H$ is given by the map $\varphi : H \rightarrow Aut N$ associated with the semidirect product $N \rtimes_{\varphi} H$. As always, $1$ is an orbit. It follows that $1 \in H$ is the only element that fixes a point in $N^{\#}= N - 1$, and so all point-stabilizers of elements of $N^{\#}$ are trivial. So for any $x \in N^{\#}$, we have $|x^H| = |H_x||x^H| = |H|$. Thus, the dimensions of the simple modules of $k^N \# k H$ are just the irreducible character degrees of $H$, together with $\underbrace{|H|, \ldots, |H|}_{r}$ for some $r\ge 1$.

Now let us consider the irreducible character degrees of $G$. They comprise the character degrees of $H$ together with the degrees of the non-principal characters induced from $N$ to $G$. It is obvious then that if $N$ is abelian, the irreducible character degrees of $G$ will match the dimensions of the simple modules of $k^N \# k H$. But this will certainly not be the case if $N$ is not abelian. What we would like to do is to swap $N$ with an abelian group $N^{*}$, of the same order, which shares some properties with $N$ such that we can define a new homomorphism $\varphi^{*} : H \hookrightarrow Aut N^{*}$ whose image is fixed-point-free. Then $N^{*} \rtimes_{\varphi^{*}} H$ would have the desired irreducible character degrees.

We claim that $N^{*}$, as in (\ref{N*}), is suitable. We shall prove this in several stages. First, note that if $f$ is fixed-point-free on $N$, then it must be fixed-point-free on any characteristic subgroup of $N$. Now recall that each term in the lower central series is a characteristic subgroup. It follows that there is a natural automorphism induced on each of the factors $\gamma_i(N) / \gamma_{i+1}(N)$; moreover it is also fixed-point-free. It is straightforward to extend this to the direct product (\ref{N*}). Finally, to see that $N^{*}$ is abelian, note that each $\gamma_i(N) / \gamma_{i+1}(N)$ is a subgroup of $Z(N/\gamma_{i+1}(N))$. \end{proof}

\begin{rmk}

It is well-known (see \cite{Mo}) that a finite-dimensional semisimple Hopf algebra over $k$ is a group Hopf algebra if, and only if, it is cocommutative. By considering the comultiplication on a bismash product $H= k^L \# k F$ (see page \pageref{commult}) we therefore see that $H$ is a group Hopf algebra if, and only if, $L$ is an abelian normal subgroup of $G$. Being a group Hopf algebra certainly implies that the underlying algebra structure is that of a group algebra, hence this is another sufficient condition for a bismash product to be isomorphic, as an algebra, to a group algebra.

\end{rmk}

\section*{Acknowledgements}

The author would especially like to thank Michael Collins for introducing him to bismash products, and also Susan Montgomery and Jan Saxl for a number of helpful discussions.


\begin{thebibliography}{99}

\bibitem{C}
{M. J. Collins.} Some bismash products that are not group algebras, \textit{J. Algebra} \textbf{316} (2007), 297--302.

\bibitem{JM}
{A. Jedwab and S. Montgomery.} Representations of some Hopf algebras associated to the symmetric group $S_n$, \textit{Algebras and Representation Theory} \textbf{12} (2008), 1--17.

\bibitem{KMM}
{Y. Kashina, G. Mason and S. Montgomery.} Computing the Frobenius-Schur indicator for abelian extensions of Hopf algebras, \textit{J. Algebra} \textbf{251} (2002), 888--913.

\bibitem{Mo}
{S. Montgomery.} \textit{Hopf Algebras and their Actions on Rings.} CBMS Lectures vol. 82 (AMS, 1993).

\bibitem{JL}
{G. James and M. Liebeck.} \textit{Representations and Characters of Groups.} (CUP, 2001).

\bibitem{N}
{D. Nikshych.} Non group-theoretical semisimple Hopf algebras from group actions on fusion categories, \textit{Selecta Math.} \textbf{14} (2008), 145--161.

\bibitem{Ta}
{M. Takeuchi.} Matched pairs of groups and bismash products of Hopf algebras, \textit{Comm. Algebra} \textbf{9} (1981), 841--882.

\bibitem{Hu}
{B. Huppert.} \textit{Endliche Gruppen I.} (Springer-Verlag, Berlin, 1967).

\bibitem{C2}
{M. J. Collins.} \textit{Representations and characters of finite groups.} Cambridge studies in advanced math. vol. 22 (CUP, 1990).

\bibitem{ISA}
{I. M. Isaacs.} \textit{Character Theory of Finite Groups.} (Academic Press, New York, 1976).

\bibitem{Go}
{D. Gorenstein.} \textit{Finite groups.} (Harper and Row, New York, 1968).
\end{thebibliography}
\end{document}